\documentclass[10pt]{amsart}
\usepackage{amscd,amssymb,hyperref}
\usepackage[PostScript=dvips]{diagrams}
\usepackage[dvips]{graphicx}

\usepackage{graphics}
\usepackage{epsfig}
\usepackage{picinpar}
\usepackage{wrapfig}
\usepackage{amsmath}
\usepackage{color}
\usepackage{pstricks,pst-node,pst-text,pst-3d}
\usepackage{pst-plot}
\usepackage{verbatim}

\newtheorem{thm}{Theorem}[section]
\newtheorem{lem}[thm]{Lemma}

\newtheorem{prop}[thm]{Proposition}

\def\square{\vbox{
      \hrule height 0.4pt
      \hbox{\vrule width 0.4pt height 5.5pt \kern 5.5pt \vrule width 0.4pt}
      \hrule height 0.4pt}}

\def\Ker{\mathrm{K er}}
\def\ch\mathrm{c h}

\newcommand{\Z}{\mathbb{Z}}

\newcommand{\C}{\ensuremath{\mathbb{C}}}
\newcommand{\R}{\ensuremath{\mathbb{R}}}

\newcommand{\calB}{\mathcal{B}}
\newcommand{\Brun}{\mathrm{Brun}}

\numberwithin{equation}{section}

\begin{document}

\newcommand{\auths}[1]{\textrm{#1},}
\newcommand{\artTitle}[1]{\textsl{#1},}
\newcommand{\jTitle}[1]{\textrm{#1}}
\newcommand{\Vol}[1]{\textbf{#1}}
\newcommand{\Year}[1]{\textrm{(#1)}}
\newcommand{\Pages}[1]{\textrm{#1}}

\author{Roman Mikhailov}
\address{Steklov Mathematical Institute, Gubkina 8, 119991 Moscow, Russia and Institute for Advanced Study, Princeton, NJ, USA}
\email{romanvm@mi.ras.ru}
\urladdr{http://www.mi.ras.ru/\~{}romanvm/pub.html}

\author{Jie Wu}
\address{Department of Mathematics, National University of Singapore, 2 Science Drive 2
Singapore 117542} \email{matwuj@nus.edu.sg}
\urladdr{www.math.nus.edu.sg/\~{}matwujie}

\thanks{$^{\dag}$ Research is supported in part by the Academic Research Fund of the
National University of Singapore.}


\title{Homotopy groups as centers of finitely presented groups}

\begin{abstract}
For every finite abelian group $A$ and $n\geq 3$, we construct a
finitely presented group defined by explicit generators and
relations, such that its center is $\pi_n(\Sigma K(A,1))$.
\end{abstract}

\maketitle

\section{Introduction}
\vspace{.5cm} It is shown in \cite{MW} that all homotopy groups of
spheres and Moore spaces can be presented as centers of certain
finitely generated groups given by explicit generators and
relations. The following question rises naturally: for which space
$X$ and $n\geq 2$, can one construct a finitely presented group
$\Gamma(X)$ defined by generators and relations, such that the
center of $\Gamma(X)$ is $\pi_n(X)$? In this paper we study this
question for suspensions of classifying spaces of finite abelian
groups. For every finite abelian group $A$ and $n\geq 2$, we
construct a finitely presented group $\mathcal J_n$, such
that\footnote{For a group $G$, we denote its center by $Z(G)$.}
$Z(\mathcal J_n)\simeq \pi_{n+1}(\Sigma K(A,1))$.

The main approach of the construction of finitely generated groups
with centers given by homotopy groups, used in \cite{Wu2} and
\cite{MW} is the following. For certain simply-connected spaces
$X$ there are simplicial group models $G_*$, for loops of $X$,
i.e. $|G_*|\simeq \Omega X$ such that the centers of components of
$G_*$ are trivial and there is a combinatorial description of
Moore boundaries $\calB G_*$. In this case, the homotopy groups
$\pi_{n+1}(X)\simeq \pi_n(G_*)$ are isomorphic to the centers of
the quotient groups $G_n/\calB G_n$. However, for all such
simplicial models in \cite{Wu2} and \cite{MW}, the Moore
boundaries $\calB G_n$ are not generated by finitely many elements
as normal subgroups of $G_n$. Recall that, for the two-dimensional
sphere $S^2$, there is a trick based on properties of braid
groups, which gives a sequence of finitely presented groups given
by generators and relations such that their centers are homotopy
groups $\pi_*(S^2)\times \mathbb Z$ \cite{LW1}. However, this
trick does not work for other spaces, since it is based on very
specific properties of Milnor's simplicial construction $F[S^1]$
\cite{Milnor}. Recall also that, for a group $G$, such that the
commutator subgroup $[G,G]$ has a trivial center, the non-abelian
tensor square $G\otimes G$ in the sense of Brown-Loday \cite{BL}
has the following property:
$$
\pi_3(\Sigma K(G,1))\simeq Z(G\otimes G).
$$
However, a generalization of this construction for higher homotopy
groups seems to be a hard problem.

The homotopy groups of the suspended spaces $\Sigma K(A,1)$ are
highly nontrivial from the point of view of computations. For
example, the homotopy groups $\pi_n(\Sigma K(\mathbb Z/2,1))$ are
known at the moment for $n\leq 6$ (see \cite{MW1}, \cite{RR}).

Our construction is as follows. Let $A$ be an abelian group. For
$n\geq 1,$ define the free product $T_n:=(\underbrace{A\times
\dots \times A}_{n\ \text{copies}})*(\underbrace{A\times \dots
\times A}_{n\ \text{copies}}).$ For $i=1,\dots, n,$ denote by
$A_i$ (resp. $A_{n+i}$) the $i$th copy of $A$ in the first (resp.
second) free summand in $T_n$. For $j=1,\dots, 2n$, $a\in A$ we
denote by $a_j$ its value in $A_j$. Define
\begin{align}
& R_1=\langle a_1, a_{n+1}\ |\ a\in A\rangle^{T_n}\label{d1}\\
& R_i=\langle a_ia_{i-1}^{-1},\ a_{n+i}a_{n+i-1}^{-1}\ |\ a\in A\rangle^{T_n},\ 1<i\leq n\label{d2}\\
& R_{n+1}=\langle a_n, a_{2n}\ |\ a\in A\rangle^{T_n}\label{d3}
\end{align}
Recall the definition of the symmetric commutator subgroup:
\begin{equation}\label{symmcomm}
[R_1,R_2,\ldots,
R_{n+1}]_S=\prod_{\sigma\in\Sigma_{n+1}}[\dots[R_{\sigma(1)},R_{\sigma(2)}],\ldots,R_{\sigma(n+1)}].
\end{equation}
Define
$$
\mathcal J_n(A):=T_n/(\gamma_{2^{n+1}}([T_n,T_n])[R_1,\dots,
R_{n+1}]_S),
$$
where, for a group $G$, $\{\gamma_i(G)\}_{i\geq 1}$ is the lower
central sequence. The main result of the paper is the following:
\begin{thm}\label{theorem}
Let $A$ be a finite abelian group. The homotopy group
$\pi_{n+1}(\Sigma K(A,1))$ is isomorphic to the center of the
polycyclic group $\mathcal J_n(A)$ for all $n\geq 2$.
\end{thm}
All polycyclic groups are finitely presented, in particular,
$\mathcal J_n(A)$ is finitely presented, moreover, it follows from
the definition, that $\mathcal J_n(A)$ is virtually nilpotent. The
group $\mathcal J_n(A)$ is obtained canonically from the
simplicial group $K(A,1)\ast K(A,1)$. The lower central series
spectral sequence from the free product $K(A,1)\ast K(A,1)$ may
give some computational information on the homotopy groups and the
group $\mathcal J_n(A)$.

There is an explicit construction of finitely presented group
whose center is $\pi_n(S^2)\times \mathbb Z$ given in~\cite{LW1}.
However there are some mistakes in the paper~\cite{LW1}. In
Section \ref{section6}, we give a brief review for the main result
in~\cite{LW1} with corrections for the mistakes. The notion of
symmetric commutator (\ref{symmcomm}) plays the central role in
this construction.

\section{Simplicial models}\label{section2}
\vspace{.5cm} Let $A$ be an abelian group. The homotopy
commutative diagram of fibre sequences
\begin{equation}\label{equation3.1}
\begin{diagram}
\Sigma K(A,1)\wedge K(A,1)&\rTo^{H} &\Sigma K(A,1)&\rTo& K(A,2)=BK(A,1)\\
\dEq&&\dTo&\textrm{pull}&\dTo>{\Delta}\\
\Sigma K(A,1)\wedge K(A,1)&\rTo &K(A,2)\vee K(A,2)&\rInto& K(A,2)\times K(A,2),\\
\end{diagram}
\end{equation}
where $H$ is the Hopf fibration, implies that there are
isomorphisms
\begin{equation}\label{wq}
\pi_n(\Sigma K(A,1)) \cong \pi_n(K(A,2)\vee K(A,2))
\end{equation}
for $n\geq 3$.

Choose the simplest simplicial model for $K(A,1)$. Applying the
inverse to the normalization functor in the sense of Dold-Kan to
the complex $A[1]$, we obtain the abelian simplicial group  $E_*$
with components
$$
E_i=(\underbrace{A\times \dots \times A}_{i\ \text{copies}}),\
i\geq 1
$$
and the property $|E_*|\simeq K(A,1)$. The face and degeneracy
maps in $E_*$ are standard, their structure follows from the
construction of the inverse to the normalization functor.

The following fact follows directly from the Whitehead Theorem
\cite{Whitehead} (see \cite{Kan-Thurston}, Proposition 4.3, for
the simplicial version of the Whitehead Theorem):
\begin{lem}\label{element}
For an abelian group $A$, there is a homotopy equivalence
$$
|E_**E_*|\simeq \Omega(K(A,2)\vee K(A,2)).
$$
Here $E_**E_*$ is the free product of two copies of the simplicial
group $E_*$.
\end{lem}

Recall that a simplicial group $G_*$ is called free if each $G_i,\
i\geq 0$ is free with a given basis, and the bases are stable
under degeneracy operations. The following result is due to Curtis
\cite{Curtis}
\begin{thm}\label{curtis}
Let $G_*$ be a connected free simplicial group, then for each
$r\geq 2$, the homomorphism of simplicial groups $G\to
G/\gamma_r(G)$ induces isomorphisms
$$
\pi_i(G)\simeq \pi_i(G/\gamma_r(G))
$$
for all $i<log_2r$.
\end{thm}

Now we consider the free product of simplicial groups
$T_*:=E_**E_*$. It has the following components
$$T_i:=(\underbrace{A\times \dots \times A}_{i\
\text{copies}})*(\underbrace{A\times \dots \times A}_{i\
\text{copies}}),\ i\geq 1.$$
\begin{lem}\label{lem1}
The simplicial group $[T_*,T_*]$ is free.
\end{lem}
\begin{proof}
For abelian groups $B,C$, the commutator subgroup $[G,G]$ of the
free product $G=B*C$, is a free group with basis given by all
commutators\footnote{We use the standard notation
$[a,b]:=a^{-1}b^{-1}ab$.} $[b,c],\ 1\neq b\in B,\ 1\neq c\in C$
(see, for example, \cite{MKS}). Taking such commutators as basis
elements in $[T_*,T_*]$, we immediately obtain from definition of
$T_*$, that the bases are stable under degeneracy operations.
\end{proof}
Now Theorem \ref{curtis} and Lemma \ref{lem1} imply that, the
natural map of simplicial groups
$$
[T_*,T_*]\to [T_*,T_*]/\gamma_{2^{n+1}}([T_*,T_*])
$$
induces isomorphism of homotopy groups
$$
\pi_i([T_*,T_*])\simeq
\pi_i([T_*,T_*]/\gamma_{2^{n+1}}([T_*,T_*])),\ i<n.
$$
The short exact sequence of simplicial groups:
$$
1\to [T_*,T_*]\to T_*\to T_*/[T_*,T_*]\to 1
$$
induces the long exact sequence of homotopy groups. The homotopy
equivalence
$$
|T_*/[T_*,T_*]|\simeq K(A,1)\times K(A,1)
$$
implies that, for $i\geq 3$, there is an isomorphism
$$
\pi_i([T_*,T_*])\simeq \pi_i(T_*).
$$
Lemma \ref{element} together with isomorphisms (\ref{wq}) imply
that, for all $2<i<n$, there are isomorphisms of homotopy groups
\begin{equation}\label{miso}
\pi_i(T_*)\simeq \pi_i(T_*/\gamma_{2^{n+1}}([T_*,T_*]))\simeq
\pi_{i+1}(\Sigma K(A,1)).
\end{equation}

\vspace{.5cm}
\section{Proof of the Main Theorem}
\vspace{.5cm}
\begin{lem}\label{lemma31}
Let $A$ and $B$ be nontrivial finite abelian groups, $G=A*B$. The
center of the group $H=G/\gamma_n([G,G])$ is trivial for $n\geq
2$.
\end{lem}
\begin{proof}
The commutator subgroup $[G,G]$ is free with a basis $\{[a,b],\
1\neq a\in A,\ 1\neq b\in B\}$ (see \cite{MKS}). We prove the
statement for $n=2$, the proof for higher $n$ is similar.

Let $h\in [G,G]$ then, modulo $\gamma_2([G,G]),$ $h$ can be
uniquely written as follows
\begin{equation}\label{form}
h=\prod_{1\neq a\in A,\ 1\neq b\in B}[a,b]^{m(a,b)},\ m(a,b)\in
\mathbb Z.
\end{equation}
For $c\in A,$ we have
$$
h^c\equiv \prod_{1\neq a\in A,\ 1\neq b\in
B}[ca,b]^{m(a,b)}[c,b]^{-m(a,b)} \mod \gamma_2([G,G])
$$
and
$$
[h,c]\equiv \prod_{1\neq a\in A,\ 1\neq b\in
B}[ca,b]^{m(a,b)}[a,b]^{-m(a,b)}[c,b]^{-m(a,b)} \mod
\gamma_2([G,G])
$$
Assume that $1\neq \alpha\in Z(H).$ We can write $\alpha$ as
$\alpha=fdh.\gamma_2([G,G]),\ f\in A,\ d\in B,\ h\in [G,G]$.
Assume that $d\neq 1$. Writing $\alpha$ in the form (\ref{form})
and taking $c\in A$, we have
$$
[c,\alpha]\equiv [c,d]\prod_{1\neq a\in A,\ 1\neq b\in
B}[ca,b]^{m(a,b)}[a,b]^{-m(a,b)}[c,b]^{-m(a,b)} \mod
\gamma_2([G,G])
$$
Since $\alpha\in Z(H)$, and $ca\neq c$ for $a\neq 1,$ we have
\begin{equation}\label{iden}
m(c,d)+\sum_{1\neq a\in A}m(a,d)=1.
\end{equation}
Since we can choose arbitrary element $c\in A$, the identity
(\ref{iden}) holds for every $c\in A$ (but $d\in B$ is fixed).
Summing up over all $c\in A$, we have
$$
(1+|A|)(\sum_{1\neq a\in A}m(a,d))=|A|
$$
but this is not possible, since all coefficients $m(a,d)\in
\mathbb Z$. Therefore, $d=1$. Analogous argument shows that $f=1$
and hence $\alpha\in [G,G].\gamma_2([G,G])$. Now we present
$\alpha=h.\gamma_2([G,G])$ in the form (\ref{form}). Since
$\alpha\in Z(H)$, we have
$$
[h,c]\equiv \prod_{1\neq a\in A,\ 1\neq b\in
B}[ca,b]^{m(a,b)}[a,b]^{-m(a,b)}[c,b]^{-m(a,b)} \equiv 0 \mod
\gamma_2([G,G])
$$
for any $c\in A$. Therefore, for every $b\in B$, we have
\begin{equation}\label{rela}
m(c,b)+\sum_{1\neq a\in A}m(a,b)=0.
\end{equation}
Again, summing up over all $c\in A$, we have
$$
(1+|A|)\sum_{1\neq a\in A}m(a,b)=0
$$
and therefore, $\sum_{1\neq a\in A}m(a,b)=0$. Now (\ref{rela})
implies that $m(c,b)=0$ for all $c\in A,\ b\in B$. Therefore,
$h\in \gamma_2([G,G])$ and the statement is proved.
\end{proof}

\begin{proof}[Proof of Theorem~\ref{theorem}]
Lemma \ref{lemma31} implies that, the centers of the components of
the simplicial group $T_*/\gamma_{2^{n+1}}([T_*,T_*])$ are
trivial. Isomorphism (\ref{miso}) together with \cite[Proposition
2.14]{Wu2} imply that there is an isomorphism
$$
\pi_{n+1}(\Sigma K(A,1))\simeq
Z(T_*/\gamma_{2^{n+1}}([T_*,T_*])\mathcal B_n)
$$
where $\mathcal B_n$ is the Moore boundary. It remains to show
that the Moore boundary is given by symmetric commutator subgroup
\begin{equation}\label{equation3.4}
\mathcal B_n=[R_1, R_2,\ldots, R_{n+1}]_S.
\end{equation}
Let $F=F(A\smallsetminus \{1\})$ be the free group free generated
by all nontrivial elements of $A$\footnote{Here we use the
multiplicative notations for the elements of abelian groups.}. Let
$\phi\colon F\to A$ be the canonical quotient homomorphism, namely
$\phi$ is the (unique) group homomorphism such that $\phi(a)=a$
for $a\in A\smallsetminus\{1\}$. Then $\phi$ induces a simplicial
epimorphism\footnote{For the description of Carlsson's
construction $F^G[S^1]$, see \cite{Carl}, \cite{Wu2}.}
$$
\tilde\phi\colon F^{F}[S^1]\twoheadrightarrow F^A[S^1] \twoheadrightarrow F^A[S^1]^{\mathrm{ab}}=K(A,1).
$$
Recall that the simplicial circle $S^1$ has the elements that can be explicitly given by
$$
S^1_n=\{\ast, x_{i+1}=s_{n-1}\cdots s_{i+1}s_i\cdots s_0\sigma_1\ | \ 0\leq i\leq n-1\},
$$
where $\sigma_1$ is the nondegenerate element in $S^1_1$, and by
the definition of Carlsson's construction \cite{Carl},
$F^F[S^1]_n$ is the self free product of $F$ index by elements in
$S^1_n\smallsetminus \{*\}$. Thus
$$
F^F[S^1]_n=(F)_{x_1}\ast (F)_{x_2}\ast\cdots\ast (F)_{x_n},
$$
where $(F)_{x_i}$ is a copy of $F$ labeled by $x_i$. The epimorphism $\tilde\phi\colon F^F[S^1]_n\to K(A,1)_n$ is explicitly given by the composite
$$
(F)_{x_1}\ast (F)_{x_2}\ast\cdots\ast (F)_{x_n} \twoheadrightarrow (A)_{x_1}\ast (A)_{x_2}\ast\cdots\ast (A)_{x_n}\twoheadrightarrow (A)_{x_1}\times (A)_{x_2}\times\cdots\times (A)_{x_n}.
$$
Thus $$\tilde\phi((a)_{x_i})=a_i$$ for $a\in A\smallsetminus\{0\}$ for $1\leq i\leq n$. Consider the epimorphism
$$
\tilde\phi\ast \tilde\phi\colon F^F[S^1]\ast F^F[S^1]\twoheadrightarrow K(A,1)\ast K(A,1).
$$
Observe that $F^F[S^1]\ast F^F[S^1]=F^{F\ast F}[S^1]$. Following
the notation in the introduction for the group $T_n$, let $B$ be a
copy of $A$ and so the group $F\ast F$ is the free group with a
basis given by $a$ for $a\in A\smallsetminus\{1\}$ and $b$ for
$b\in B\smallsetminus\{1\}$. The epimorphism
$$\tilde\phi\ast\tilde\phi\colon F^{F\ast F}[S^1]_n\to K(A,1)_n\ast K(A,1)_n=T_n$$ is given by
$$
(\tilde\phi\ast\tilde \phi)((a)_{x_i})=a\in A_i \textrm{ and } (\tilde\phi\ast\tilde \phi)((b)_{x_i})=b\in B_i
$$
for $1\leq i\leq n$. Let
\begin{align*}
& \tilde R_1=\langle (a)_{x_1}, (b)_{x_1} \ | \ a \in A \smallsetminus\{1\},b\in B\smallsetminus\{1\} \rangle^{F^{F\ast F}[S^1]_n}\\
& \tilde R_i=\langle (a)_{x_i}(a)_{x_{i-1}}^{-1},\ (b)_{x_i}(b)_{x_{i-1}}^{-1}\ |\  a \in A \smallsetminus\{1\},b\in B\smallsetminus\{1\}\rangle^{F^{F\ast F}[S^1]_n},\ 1<i\leq n\\
& \tilde R_{n+1}=\langle (a)_{x_n}, (b)_{x_n} \ | \ a\in
A\smallsetminus\{1\}, b\in B\smallsetminus\{1\}\rangle^{F^{F\ast
F}[S^1]_n}.
\end{align*}
Then
$$
\tilde\phi\ast\tilde\phi(\tilde R_i)=R_i
$$
for $1\leq i\leq n$ and so
\begin{equation}\label{equation3.5}
\tilde\phi\ast\tilde\phi([\tilde R_1,\tilde R_2,\ldots,
\tilde R_{n+1}]_S)=[R_1,R_2,\ldots,
R_{n+1}]_S.
\end{equation}
Let $H_{j}$ be a sequence of subgroups of $G$ for $1\leq j\leq k$. Recall that the fat commutator subgroup
$[[H_{1},\dots,H_{k}]]$ of $G$ is generated by all of
the commutators $\beta^t(h_{i_1}^{(1)},\dots,h_{i_t}^{(t)})$,
where
\begin{enumerate}
\item[1)] $1\leq i_s\leq k$;
\item[2)] all integers in $\{1,2,\cdots,k\}$ appear as at least one of the
integers $i_s$;
\item[3)] $h_j^{(s)}\in H_j$;
\item[4)] for each $t\geq k$,
 $\beta^t$ runs over all of the  bracket arrangements of weight $t$.
 \end{enumerate}
By~\cite[Proof of Theorem 1.8]{Wu2}, the Moore boundary
$$
\calB_nF^{F\ast F}[S^1]=[[\tilde R_1, \tilde R_2,\ldots, \tilde R_{n+1}]].
$$
According to~\cite[Theorem 1.1]{LW2},
$$
[[\tilde R_1, \tilde R_2,\ldots, \tilde R_{n+1}]]=[\tilde R_1, \tilde R_2,\ldots, \tilde R_{n+1}]_S
$$
It follows that
\begin{equation}\label{equation3.6}
\calB_nF^{F\ast F}[S^1]=[\tilde R_1, \tilde R_2,\ldots, \tilde R_{n+1}]_S.
\end{equation}
By~\cite[Lemma 5, 3.8]{Quillen}, any simplicial epimorphism induces an epimorphism on the Moore chains and so
$$
\begin{array}{rcl}
\calB_n (K(A,1)\ast K(A,1))&=&d_0N_{n+1}(K(A,1)\ast K(A,1))\\
&=&d_0(\tilde\phi\ast\tilde\phi(N_{n+1}(F^{F\ast}[S^1]))\\
&=&\tilde\phi\ast\tilde\phi(d_0N_{n+1}(F^{F\ast F}[S^1]))\\
&=&\tilde\phi\ast\tilde\phi(\calB_nF^{F\ast F}[S^1]).\\
\end{array}
$$
Equation~(\ref{equation3.4}) follows from
equations~(\ref{equation3.5}) and~(\ref{equation3.6}) now. This
finishes the proof. Observe that, for every pair $A,B$ of finite
abelian groups, the quotient group $A*B/\gamma_n([A*B,A*B])$ is
polycyclic, i.e. it is a solvable group such that every its
subgroup is finitely-generated. In particular, the groups
$\mathcal J_n(A)=T_n/\gamma_{2^{n+1}}([T_n,T_n])\mathcal B_n$ are
finitely presented for all $n\geq 2$.
\end{proof}
\vspace{.5cm}
\section{Homotopy Groups of $S^2$}\label{section6}
\vspace{.5cm} Taking the simplicial group $E_**E_*$ from Section
\ref{section2}, for $A=\mathbb Z$, we get the homotopy equivalence
$$
|E_**E_*|\simeq \Omega (K(\mathbb Z,2)\vee K(\mathbb Z,2))
$$
Since the Moore boundaries of $E_**E_*$ can be described
combinatorially using symmetric commutators, we obtain the
following description of homotopy groups of $S^2$ alternative to
the description given in \cite{Wu2}.
\begin{prop}
Let $n\geq 3$, $T_n:=(\underbrace{\mathbb Z\times \dots \times
\mathbb Z}_{n\ \text{copies}})*(\underbrace{\mathbb Z\times \dots
\times \mathbb Z}_{n\ \text{copies}}).$ There is an isomorphism
$$
\pi_{n+1}(S^2)\simeq Z(T_n/[R_1,\dots, R_{n+1}]_S),
$$
where the subgroups $R_i,\ i=1,\dots, n+1$ are defined as in
(\ref{d1})-(\ref{d3}).
\end{prop}

Observe that, the groups $T_n/[R_1,\dots, R_{n+1}]_S$ are not
finitely presented as well as the groups from \cite{Wu2} whose
center is $\pi_{n+1}(S^2)$. There is an explicit construction of
finitely presented group whose center is $\pi_n(S^2)\times \mathbb
Z$ given in~\cite{LW1}. However there are some mistakes in the
paper~\cite{LW1}. In this section, we give a brief review for the
main result in~\cite{LW1} with corrections for the mistakes.

Let $d_i\colon P_n\to P_{n-1}$ be the group homomorphism given by
removing the $i$th strand of $n$-strand pure braids for $1\leq
i\leq n$. There exists a well-defined additional face operation
$d_0\colon P_n\to P_{n-1}$ as a group homomorphism defined by
$$
d_0A_{i,j}=A_{i-1,j-1}
$$
for $1\leq i<j\leq n$, where the braid $A_{0,j}\in P_{n-1}$ is given by
$$
\begin{array}{rcl}
A_{0,j}  &  =&(A_{j,j+1}A_{j,j+2}\cdots A_{j,n-1})^{-1}(A_{1,j}\cdots
A_{j-1,j})^{-1}\\
&  =&(\sigma_{j}\cdots\sigma_{n-3}\sigma_{n-2}^{2}\sigma
_{n-3}\cdots\sigma_{j})^{-1}\cdot(\sigma_{j-1}\cdots\sigma_{2}\sigma_{1}^{2}\sigma_{2}\cdots\sigma_{j-1})^{-1}.\\
\end{array}
$$
Then the sequence of groups $\mathbb{P}=\{P_n\}_{n\geq 0}$ forms a $\Delta$-group with the property that~\cite[Proposition 2.5]{LW1} the Moore homotopy group of the $\Delta$-group $\mathbb{P}$ given by
\begin{equation}\label{equation6.1}
\pi_n(\mathbb{P})\cong \pi_n(S^2).
\end{equation}
From the definition of the Moore homotopy groups of $\Delta$-groups, the Moore chains
\begin{equation}\label{equation6.2}
N_n\mathbb{P}=\bigcap_{i=1}^n \Ker(d_i\colon P_n\to P_{n-1})=\Brun_n,
\end{equation}
where $\Brun_n$ is the group of $n$-strand pure Brunnian braids. Let
$$
\mathrm{Bd}_n=d_0(\Brun_{n+1})=\calB_n\mathbb{P}
$$
be the Moore boundaries of $\mathbb{P}$, which is called
$n$-strand \textit{boundary Brunnian braids} in~\cite{LW1}. The
main result of~\cite{LW1} is as follows.

\begin{thm}~\cite[Theorem 1.1 (2)]{LW1}\label{theorem6.1}
The group $\mathrm{Bd}_n$ is a normal subgroup of the Artin braid group $B_n$. There are isomorphisms of groups
$$
Z(P_n/\mathrm{Bd}_n)\cong \pi_n(S^2)\times \Z\textrm{ and } Z(B_n/\mathrm{Bd}_n)\cong \{\alpha\in \pi_n(S^2)\ | \ 2\alpha=0\}\times\Z
$$
for $n\geq 4$, where $\Z$ is induced from
$Z(P_n)=Z(B_n)=\Z$.\hfill $\Box$
\end{thm}

The main results~\cite[Theorem 1, Theorem 3]{LW1} are correct.
Also it is correct that the groups $P_n/\mathrm{Bd}_n$ and
$B_n/\mathrm{Bd}_n$ are finitely presented. The major mistake
in~\cite{LW1} is that~\cite[Lemma 3.6]{LW1} is not correct. As a
consequence of this false lemma, the statements~\cite[Theorem 3.7,
Lemma 3.11, Corollary 3.12, Corollary 3.13, Corollary 3.14,
Proposition 3.15]{LW1} are not true. For the introduction to the
main results in~\cite{LW1}, the set of normal generators for
$\Brun_n$ in~\cite[line 3, p.522]{LW1} and for $\mathrm{Bd}_n$
in~\cite[line 14, p.523]{LW1} are not correct. In order to correct
the mistakes, we determine a finite set of normal generators for
$\Brun_n$ as well as a finite set of normal generators for
$\mathrm{Bd}_n$. This will confirm that $P_n/\mathrm{Bd}_n$ and
$B_n/\mathrm{Bd}_n$ are finitely presented.

\begin{thm}
Let $\Brun_n$ be the group of $n$-strand Brunnian pure braids and let $\mathrm{Bd}_n$ be the group of $n$-strand boundary Brunnian pure braids. Then
\begin{enumerate}
\item[1)] The group $\Brun_n$ is the normal closure of the following elements in $P_n$:
$$
[[[A_{1,n}, A_{\sigma(2),j_2}], A_{\sigma(3),j_3}],\ldots,A_{\sigma(n-1),j_{n-1}}]
$$
for $\sigma\in\Sigma_{n-2}$ acting on $\{2,3,\ldots,n-1\}$ and $1\leq j_2,\ldots,j_{n-1}\leq n$ with $j_s\not=\sigma(s)$ for $2\leq s\leq n-1$, where $A_{j,i}=A_{i,j}$ if $i>j$.
\item[2)] The group $\mathrm{Bd}_n$ is the normal closure of the following elements in $P_n$:
$$
[[[A_{0,n}, A_{\sigma(1),j_1}], A_{\sigma(2),j_2}],\ldots,A_{\sigma(n-1),j_{n-1}}]
$$
for $\sigma\in\Sigma_{n-1}$ and $1\leq j_1,\ldots,j_{n-1}\leq n$ with $j_s\not=\sigma(s)$ for $2\leq s\leq n-1$, where $A_{j,i}=A_{i,j}$ if $i>j$.
\end{enumerate}
\end{thm}
\begin{proof}
(1). Let $R_{i,j}$ be the normal closure of $A_{i,j}$ in $P_n$ for $1\leq i<j\leq n$. From~\cite[Theorem 1.1]{BMVW},
$$
\Brun_n=\prod_{\sigma\in\Sigma_{n-1}}[[R_{\sigma(1),n},R_{\sigma(2),n}],\ldots,R_{\sigma(n-1),n}].
$$
By~\cite[Theorem 1.2]{LW2},
$$
\prod_{\sigma\in\Sigma_{n-1}}[[R_{\sigma(1),n},R_{\sigma(2),n}],\ldots,R_{\sigma(n-1),n}]=\prod_{\sigma\in\Sigma_{n-2}}[[R_{1,n},R_{\sigma(2),n}],\ldots,R_{\sigma(n-1),n}].
$$
Thus
\begin{equation}\label{equation6.3}
\Brun_n=\prod_{\sigma\in\Sigma_{n-2}}[[R_{1,n},R_{\sigma(2),n}],\ldots,R_{\sigma(n-1),n}].
\end{equation}
For each $1\leq i\leq n$, let $G_i$ be the subgroup of $P_n$ generated by $A_{i,j}$ for $1\leq j\leq n$ with $j\not=i$, where $A_{i,j}=A_{j,i}$ if $i>j$. Recall that $P_n=\pi_1(F(\C,n))$, where
\begin{equation}\label{equation6.4}
F(\C,n)=\{(z_1,\ldots,z_n)\ | \ z_i\not=z_j \textrm{ for } i\not=j\}
\end{equation}
is the ordered configuration space. By~\cite[Proposition
3.2]{BMVW}, the removal of the $i$th strand operation $d_i\colon
P_n\to P_{n-1}$ is induced by the coordinate projection
$$
\pi_i\colon F(\R^2,n)\longrightarrow F(\R^2,n-1)\quad (z_1,z_2,\ldots,z_n)\mapsto (z_1,\ldots,z_{i-1},z_{i+1},\ldots,z_n).
$$
Let $(p_1,\ldots,p_n)$ be a base-point of $F(\C,n)$. Namely
$p_1,\ldots,p_n$ are $n$ distinct points in the plane $\R^2$. By
the classical Fadell-Neuwirth Theorem~\cite{FN}, the coordinate
projection $\pi_i$ is a fibration with a punctured plane
$\R^2\smallsetminus \{p_1,\ldots, p_{i-1},p_{i+1},\ldots,p_n\}$ as
a fibre. By taking the fundamental groups to the
fibration~(\ref{equation6.4}), we obtain $$\Ker(d_i\colon P_n\to
P_{n-1})=G_i=F_{n-1}.$$
Moreover for $1\leq i<j\leq n$
$$
G_i\cap G_j=\Ker(d_i\colon P_n\to P_{n-1})\cap \Ker(d_j\colon P_n\to P_{n-1})=R_{i,j}
$$
since the restriction of $d_i$ to $G_j$ is given by the projection
map
$$
F(A_{1,j},\ldots, A_{j-1,j},A_{j,j+1},\ldots, A_{j,n})\longrightarrow F(A_{1,j-1},\ldots,A_{j-1,n-1})\leq P_{n-1}
$$
with $d_iA_{i,j}=1$, $d_iA_{s,j}=A_{s,j-1}$ for $s<i$.
$d_iA_{s,j}=A_{s-1,j-1}$ for $i<s<j$ and $d_iA_{j,s}=A_{j-1,s-1}$
for $s>j$. Consider the factors in product~(\ref{equation6.3}).
For each $\sigma\in \Sigma_{n-2}$,
$$
[[R_{1,n},R_{\sigma(2),n}],\ldots,R_{\sigma(n-1),n}]\leq [[R_{1,n},G_{\sigma(2)}],\ldots,G_{\sigma(n-1)}].
$$
Observe that
$$
[[R_{1,n},G_{\sigma(2)}],\ldots,G_{\sigma(n-1)}]\leq\Brun_n.
$$
We can construct a finite set of normal generators in $P_n$ for $[[R_{1,n},G_{\sigma(2)}],\ldots,G_{\sigma(n-1)}]$ based on the following statement:

\bigskip

\noindent\textbf{Statement 4.4}~\cite[Proof of Lemma 5.2]{BMVW}.
Let $G$ be a group and let $A$ and $B$ be normal subgroups of $G$.
If $\{a_i\ | \ i\in I\}$ is a set of normal generators for $A$ in
$G$ and $\{b_j\ | \ j\in J\}$ is a set of generators for $B$, then
$\{[a_i,b_j]\ | \ i\in I, j\in J\}$ is a set of normal generators
for the commutator subgroup $[A,B]$ in $G$.

\bigskip

The construction of a set of normal generators for
$[[R_{1,n},G_{\sigma(2)}],\ldots,G_{\sigma(n-1)}]$ is as follows.
Observe that $R_{1,n}$ has a normal generator $A_{1,n}$ and
$G_{\sigma(2)}$ has generators $A_{\sigma(2),j_2}$ for $1\leq
j_2\leq n$ with $j_2\not=\sigma(2)$. From the above statement, a
set of normal generators for $[R_{1,n}, G_{\sigma(2)}]$ is given
by
$$
[A_{1,n},A_{\sigma(2), j_2}]
$$
for $1\leq j_2\leq n$ with $j_2\not=\sigma(2)$. By repeating this procedure, a set of normal generators for $[[R_{1,n},G_{\sigma(2)}],\ldots,G_{\sigma(n-1)}]$ is given by
$$
[[[A_{1,n}, A_{\sigma(2),j_2}], A_{\sigma(3),j_3}],\ldots,A_{\sigma(n-1),j_{n-1}}]
$$
for $1\leq j_s\leq n$ with $j_s\not=\sigma(s)$. This finishes the proof of assertion (1).

(2). By definition, $\mathrm{Bd}_n=d_0(\Brun_{n+1})$. Since
$$
d_0([[A_{1,n+1}, A_{\sigma(1)+1,j_1+1}],\ldots,A_{\sigma(n-1)+1,j_{n-1}+1}])=[[A_{0,n}, A_{\sigma(1),j_1}],\ldots,A_{\sigma(n-1),j_{n-1}}]
$$
with $[[A_{1,n+1}, A_{\sigma(1)+1,j_1+1}],\ldots,A_{\sigma(n-1)+1,j_{n-1}+1}]\in \Brun_{n+1}$, the elements listed in assertion (2) lie in $\mathrm{Bd}_n$. Now from equation~\ref{equation6.3}, we have
\begin{equation}\label{equation6.5}
\begin{array}{rcl}
\mathrm{Bd}_n&=&\prod\limits_{\sigma\in\Sigma_{n-1}}[[d_0(R_{1,n+1}),d_0(R_{\sigma(2),n+1})],\ldots,d_0(R_{\sigma(n),n+1})]\\
&=&\prod\limits_{\sigma\in\Sigma_{n-1}}[[R_{0,n},R_{\sigma(2)-1,n}],\ldots,R_{\sigma(n)-1,n}],\\
\end{array}
\end{equation}
where $R_{0,j}$ is the normal closure of $A_{0,j}$ in $P_n$. Given any $\sigma\in \Sigma_{n-1}$ acting on $\{2,\ldots,n\}$, the factor
$$
[[R_{0,n},R_{\sigma(2)-1,n}],\ldots,R_{\sigma(n)-1,n}]\leq
[[R_{0,n},G_{\sigma(2)-1}],\ldots,G_{\sigma(n)-1}]
$$
with $1\leq \sigma(t)-1\leq n-1$ for $2\leq t\leq n$. From
Statement 4.4, a set of normal generators for
$[[R_{0,n},G_{\sigma(2)-1}],\ldots,G_{\sigma(n)-1}]$ is given by
$$
[[[A_{0,n},A_{\sigma(2)-1, j_2}],A_{\sigma(3)-1,j_3}],\ldots, A_{\sigma(n)-1, j_{n}}]
$$
with $1\leq j_2,j_3,\ldots,j_{n}\leq n$ and $j_s\not=\sigma(s)-1$ for $2\leq s\leq n$. Assertion (2) follows.
\end{proof}

\end{document}